\documentclass{amsart}
\usepackage{amscd,amsmath,amssymb,amsfonts,verbatim, xcolor,graphicx}
\usepackage[all]{xy}
\allowdisplaybreaks
\usepackage{calrsfs}
\DeclareMathAlphabet{\mathcal}{OMS}{zplm}{m}{n}
\makeatletter
\newcommand*\dotp{\mathpalette\dotp@{.5}}
\newcommand*\dotp@[2]{\mathbin{\vcenter{\hbox{\scalebox{#2}{$\m@th#1\bullet$}}}}}
\makeatother

\makeatletter
\newcommand*\bigcdot{\mathpalette\bigcdot@{.5}}
\newcommand*\bigcdot@[2]{\mathbin{\vcenter{\hbox{\scalebox{#2}{$\m@th#1\bullet$}}}}}
\makeatother

\newtheorem{theorem}{Theorem}[subsection]
\newtheorem{corollary}[theorem]{Corollary}
\newtheorem{lemma}[theorem]{Lemma}  
\newtheorem{proposition}[theorem]{Proposition}
\theoremstyle{definition}

\newtheorem*{conjecture*}{Conjecture}
\newtheorem*{remark*}{Remark}
\newtheorem{example}[theorem]{Example}

\numberwithin{equation}{subsection}

\newcommand{\M}{\mathcal{M}}

\title{ Infinitesimal  Bloch Regulator}

\begin{document}
\author{S\.{I}nan \"{U}nver}
\address{Ko\c{c} University, Mathematics Department. Rumelifeneri Yolu, 34450, Istanbul, Turkey}
\email{sunver@ku.edu.tr}
\maketitle
\noindent

\begin{abstract}
In this paper, we continue our project of defining and studying the  infinitesimal versions of the classical, real analytic, invariants of motives. Here, we construct an infinitesimal analog of Bloch's regulator. Let $X/k$ be a scheme of finite type over a field $k$ of characteristic 0. Suppose that  $\underline{X} \hookrightarrow X$  is a closed subscheme, smooth over $k,$ and defined by a square-zero sheaf of ideals, which is locally free on $\underline{X}.$ We define two regulators: $\rho_1,$ from  the infinitesimal part of the motivic cohomology ${\rm H}^2 _{\M}(X,\mathbb{Q}(2))$ of $X$ to ${\rm ker}({\rm H}^{0}(X,\Omega^{1} _{X}/d\mathcal{O}_{X}) \to {\rm H}^{0}(\underline{X},\Omega^{1} _{\underline{X}}/d\mathcal{O}_{\underline{X}});$ and $\rho_2,$ from ${\rm ker}(\rho_1)$ to ${\rm H}^1(X,D_{1}(\mathcal{O}_{X})),$ where $D_{1}(\mathcal{O}_X)$ is the Zariski sheaf associated to the first  Andr\'{e}-Quillen homology. The main tool is a generalization of our additive dilogarithm construction. Using Goodwillie's theorem, we deduce that $\rho_2$ is an isomorphism.  We also reinterpret the above results in terms of the infinitesimal Deligne-Vologodsky crystalline complex $\mathcal{D}_{X}^{\circ}(2),$  when $X$ is smooth over the dual numbers of $k.$
\end{abstract}

\section{Introduction}

Let   $X/\mathbb{C}$ be a smooth curve over the complex numbers. The regulator map from the motivic cohomology ${\rm H}^{2} _{M}(X,\mathbb{Q}(2))=K_{2}(X)_{\mathbb{Q}} ^{(2)}$ to analytic Deligne cohomology ${\rm H}^{2} _{D}(X_{an},\mathbb{Z}(2))\simeq{\rm H}^1(X_{an},\mathbb{C}^{\times})$ is fundamental both in the arithmetic \cite{ram} study of $X,$ when it is defined over a number field, and in the geometric study of $X$ \cite{gg}.   The construction associates to every pair $f,g$ of meromorphic functions on $X,$ a line bundle with connection on $X_{an},$  such that the monodromy at each point is given by the tame symbol of $f$ and $g$ at that point \cite{de}. Using the identification  ${\rm H}^1(X_{an} ',\mathbb{C}^{\times})={\rm Hom}({\rm H}_{1}(X_{an} ',\mathbb{Z}),\mathbb{C}^{\times}),$ where $X_{an} '$ is the open set where $f$ and $g$ are invertible, this line bundle with connection corresponds to the homomorphism that sends the closed path $\gamma$ to  

$$
\exp(\frac{1}{2\pi i}( \int_{\gamma} \log f \cdot   dlog (g)-\log g(p) \int_{\gamma} dlog (f))) 
$$
in $\mathbb{C}^{\times}$ \cite{ram}. Here $p$ is an arbitrary point on $X'_{an}$ and the construction is independent of the choice of $p.$ Since, by the Gersten resolution, $K_{2}(X)^{(2)}_{\mathbb{Q}}=\Gamma(X, K_{2}^{M}(\mathcal{O}_{X})_{\mathbb{Q}} ),$ the above construction gives the regulator from $K_{2}(X)_{\mathbb{Q}} ^{(2)}.$

The aim of the present paper is to give a precise  infinitesimal analog of this construction. In order to do this, we first need to define the correct infinitesimal verison of the motivic cohomology group and interpret it in terms of a Zariski sheaf. The global sections of the sheafification of $K_{2} ^{M}$  gives only a quotient of the correct cohomology group. Instead, we need need to consider the Bloch complex of weight two, one of whose cohomology groups is $K_{2} ^{M},$ but also has another non-trivial cohomology group. We define these motivic cohomology groups in weights 1 and 2 in \textsection 2. The basic set-up is a scheme $X/k,$ over a field $k$ of characteristic 0,  together with a square-zero sheaf of ideals such that the corresponding closed scheme $\underline{X}$ is smooth over $k.$ We justify this definition by relating it to the infinitesimal part of the $K$-theory of $X,$ when  $X$ is smooth over $k_2:=k[t]/(t^2).$ This is done by comparing this construction to the Deligne-Vologodsky complex in \textsection \ref{vol}.  

The infinitesimal complex we define is denoted by $\Gamma_{X} ^{\circ}(2).$ The corresponding hypercohomology group ${\rm H}^2(X,\Gamma_{X}^{\circ}(2)) $ is the analog of the group $K_{2}(X)_{\mathbb{Q}} ^{(2)}$ in the classical case. This very explicit and function theoretic complex allows us to construct the infinitesimal analog of the classical function theoretic approach above as follows.

Suppose that $A$ is a local  $k$-algebra together with a square-zero ideal $I$ such that $\underline{A}:=A/I$ is a smooth $k$-algebra. Then corresponding to each splitting $\tau$ of the surjection $A \to \underline{A}$ as $k$-algebras, we construct, in \textsection \ref{reg}, a regulator $\ell i_{2,\tau}$ from $B_{2}(A)$ to $D_{1}(A),$ the Andr\'{e}-Quillen homology of $A.$ We would like to think of the choice of a splitting as the analog of the  choice of a path in the classical analytic theory. We give two different but equivalent constructions, one  computational, the other one conceptual.

In order to globalize this construction, we need to compare  different choices of liftings.  We give an example that this comparison is not possible unless we impose the additional hypothesis that $I$ is a free $\underline{A}$-module. Before making this homotopy construction, we study the local structure of $\Gamma_{X} ^{\circ}(2)$ in detail in \textsection \ref{loc-bloch}. In \textsection \ref{homotopy section}, we make this construction for an arbitrary morphism of pairs of rings with a nilpotent ideal. This generality will allow us to deduce the functoriality of the constructions. Based on the above analogy, this section should be thought of as the study of what happens when one chooses a different path of integration. In \textsection \ref{explicit hom}, we give an explicit formula for this homotopy map. 

Letting $(\Omega ^{1} _{X}/d\mathcal{O}_{X})^{\circ}:={\rm ker }(\Omega ^{1} _{X}/d\mathcal{O}_{X}\to \Omega ^{1} _{\underline{X}}/d\mathcal{O}_{\underline{X}})$ denote the infinitesimal part of $\Omega ^{1} _{X}/d\mathcal{O}_{X}$ and $F\Gamma_{X} ^{\circ}(2)$ an appropriate subcomplex of $\Gamma_{X} ^{\circ}(2)$ defined in \textsection 2, we have the following main theorem:


\begin{theorem}\label{thm1}
Suppose that  $X/k$ is a scheme of finite type over a field $k$ of characterictic 0. Suppose that $\underline{X} \hookrightarrow X$ is a closed subscheme of $X$ defined by a square-zero sheaf of ideals. Suppose further that  $\underline{X}$ is smooth over $k$ and that  the conormal sheaf of $\underline{X}$ in $X$ is locally free. Then we have the following regulators:

$$
\rho_1: {\rm H}^{2}(X,\Gamma_{X} ^{\circ}(2)) \to   {\rm H}^{0}(X,(\Omega^{1} _{X}/d\mathcal{O}_{X})^{\circ})
$$
and  
$$
\rho_2: {\rm H}^{2}(X,F\Gamma_{X} ^{\circ}(2)) \to {\rm H}^1(X,D_{1}(\mathcal{O}_X)),
$$
such that ${\rm ker}(\rho_1)={\rm H}^{2}(X,F\Gamma_{X} ^{\circ}(2))$ and $\rho_2$ is an isomorphism. These maps are functorial for arbitrary morphisms of $k$-schemes. 
\end{theorem}


This can be thought of as the infinitesimal version of the injectivity conjecture for the Bloch regulator \cite[Conjecture 1.1]{gg}. Here, we should emphasize that the construction of $\rho_1$ is immediate, the content of the theorem is in the construction of $\rho_2$ and proving that it is an isomorphism. We would like to emphasize that the map above is completely explicit just as in the classical case. 

We leave the question of finding the infinitesimal version of the above analytic bundle with connection to future work.  This requires finding the right topology to define this object and  is not directly related to the contents of this paper.  Unlike the classical case above, in the infinitesimal case, we do not need to restrict ourselves to the case of curves. A heuristic argument,  based on the conjectural Bloch-Beilinson filtration, on  why we can get away without this restriction is as follows.  In the classical case, for curves, the first graded quotient ${\rm Ext}^{0} _{\mathcal{M}}(\mathbb{Q}(0), {\rm H}_{\mathcal{M}}^2(X/\mathbb{Q},\mathbb{Q})(2))$ is 0, whereas this may not be true for higher dimensions. On the other hand, in the infinitesimal setting, this quotient is always 0 regardless of the dimension. 

Finally, we would like to remark that this paper is part of a project of defining and studying the  infinitesimal versions of classical regulator constructions, which was started in \cite{unv2}.
 
{\bf Notation.} Unless stated otherwise, all the schemes as well as  the K\"{a}hler differentials, crystalline cohomology, cyclic homology, Hochschild homology and Andr\'{e}-Quillen homology are relative to $\mathbb{Q}.$ We will always use motivic cohomology with $\mathbb{Q}$-coefficients. Therefore, we always tensor all the groups in a Bloch complex with $\mathbb{Q}$ even though our notation might not reflect this. For example, $\Lambda^2 A^{\times}:=(\Lambda^2 _{\mathbb{Z}} A^{\times})_{\mathbb{Q}}.$   
 For an $A$-module $I,$  $S^{\bigcdot} _{A}  I$ denotes the symmetric algebra of $M$ over $A.$ We remove the subscript $A$ in this notation, if it is fixed in the context.    For a ring $A,$ $A^{\flat}$ denotes the set of all units $a$ in $A$ such that $1-a$ is also a unit.  For a functor $F$ from the category of  pairs $(R,I)$ of rings $R$ and nilpotent ideals $I$ to an abelian category,  we let $F^{\circ}(R,I)$ denote the kernel of the map from $F(R,I)$ to $F(R/I,0).$ We informally refer to this object as the {\it infinitesimal part of} $F.$ We have the corresponding notion for the category of artin local algebras over a field, since their maximal ideals are nilpotent.

\section{The Infinitesimal weight two motivic cohomology} 
Fix a field $k$ of characteristic 0. 
Suppose that $X/k$ is an  scheme of finite type over $k,$ and $\mathcal{I} \subseteq \mathcal{O}_{X},$ $\mathcal{I}^2$=0, a square-zero ideal, such that if $\underline{X}$  denotes the closed subscheme defined by $\mathcal{I}$ in $X,$ then $\underline{X}$ is  a smooth variety over $k.$ These assumptions imply that the imbedding $\underline{X} \hookrightarrow X$ is Zariski locally  split \cite[Proposition 4.4]{def}. In this section, we will define the candidate for the weight two infinitesimal motivic cohomology of $X.$

\subsection{Weight one infinitesimal motivic complex} First, we will start with the trivial case of weight one. For a regular scheme $Y,$ the weight one motivic complex $\Gamma_{Y}(1)$ is quasi-isomorphic to $\mathcal{O}_{Y} ^{\times}[-1].$ We define the complex $(1+\mathcal{I})[-1]$  of Zariski sheaves on $X,$ which is quasi-isomorphic to the cone of   the map $\mathcal{O}_{X} ^{\times}[-1] \to \mathcal{O}_{\underline{X}} ^{\times}[-1]  $ as the weight one infinitesimal motivic cohomology complex $\Gamma_{X} ^{\circ}(1).$  We define 
$$
{\rm H}^{i} _{\M}(X,\mathbb{Q}^{\circ}(1)):={\rm H}^{i}(X,\Gamma_{X} ^{\circ}(1)).
$$

\subsection{Weight two infinitesimal motivic complex} 

For a ring  $A,$  let $B_{2}(A)$ denote the $\mathbb{Q}$-space generated by $[x],$ with $x \in A^{\flat},$ subject to the relations 
\begin{eqnarray*}
 [x]-[y]+[y/x]-[(1-x^{-1})/(1-y^{-1})]+[(1-x)/(1-y)]=0,
 \end{eqnarray*}
  for  $x,$ $y \in A^{\times}$ such that $(1-x)(1-y)(1-x/y) \in A^{\times}.$ Let  $\Gamma_{A}(2)$ denote the Bloch complex: 
 $$
B_{2}(A) \xrightarrow{\delta} \Lambda^2 A^{\times},
$$
where $\delta([a]):=(1-a)\wedge a$ and $B_{2}(A)$ is in degree 1.  This complex was considered for local rings in \cite{unv1}. 

If $I\subseteq A$ is a nilpotent ideal, we let $\underline{A}:=A/I.$ Since the map from $\Gamma_{A}(2)$ to $\Gamma_{\underline{A}}(2)$  is surjective, its cone is quasi-isomorphic to 
$$
B_{2} ^{\circ}(A) \xrightarrow{\delta^{\circ}} (\Lambda^{2}A^{\times})^{\circ},
$$
which we denote by   $\Gamma^{\circ}_{A}(2).$ Note that this complex in fact depends on $I,$ but we suppress it from the notation since $I$ will always  be clear from the context. 
If $A$ is also a local ring, the cokernel of $\delta^{\circ}$ is $K_{2} ^{M} (A)^{\circ} \simeq (\Omega^{1} _{A}/dA)^{\circ}.$ The composition of the map from $(\Lambda^{2}A^{\times})^{\circ}$ to 
$(\Omega^{1} _{A}/dA)^{\circ}$ is given by $\log dlog,$ which sends an element $a\wedge b,$ with $a \in 1+I$ and $b$ in $A^{\times}$ to $\log (a)\frac{db}{b}.$  We let  $F((\Lambda^{2}A^{\times})^{\circ}):={\rm im} (\delta^{\circ}),$ and $F\Gamma_{A} ^{\circ}(2)$ the subcomplex of $\Gamma_{A} ^{\circ}(2)$ which agrees with it in degree 1 and is equal to  $F((\Lambda^{2}A^{\times})^{\circ})$ in degree 2. Sheafifying this,  we obtain the sheaves of complexes $\Gamma_{X}(2), $ $\Gamma_{X} ^{\circ}(2)$ and $F\Gamma_{X} ^{\circ}(2)$ in the Zariski topology. We define 
$$
{\rm H}^{i} _{\M}(X,\mathbb{Q}^{\circ}(2)):={\rm H}^{i}(X,\Gamma_{X} ^{\circ}(2)).
$$

\section{Regulator to Andr\'{e}-Quillen homology}\label{reg}

\subsection{Andr\'{e}-Quillen homology.}\label{AQ} We refer to \cite{lod} as a general reference for this section. 
Let $R$ be any commutative ring with unity and  $A$  an $R$-algebra. Let $P_{*}$ be a free simplicial  $R$-algebra which is a resolution of  $A.$  Then the cotangent complex $\mathbb{L}_{*}(A|R)$ of $A$ over $R$ is the complex, in the derived category of complexes of $A$-modules, associated to the simplicial $A$-module whose degree $n$ object is given by  $\mathbb{L}_{n}(A|R):=\Omega^1 _{P_{n}/R} \otimes _{P_n}A.$ Andr\'{e}-Quillen homology of $A$ over $R$ with coefficients in an $A$-module $M$ is then given as the homology of    $\mathbb{L}_{*}(A|R)\otimes _{A}M:$ 
$$
D_{*}(A|R,M):={\rm H}_{*}(\mathbb{L}_{*}(A|R)
\otimes _A M).
$$
We denote $D_{*}(A|R,A)$ by $D_{*}(A|R).$ 

Suppose from now on that $R=\mathbb{Q}.$ We then denote $D_{*}(A|\mathbb{Q},M)$ by $D_{*}(A,M)$ and  $D_{*}(A|\mathbb{Q})$ by $D_{*}(A).$ If $A$ is a quotient of a smooth algebra $B$ with kernel $J,$ then 
 the transitivity long exact sequence for  $\mathbb{Q} \to B \to A$  gives  
$$
\cdots \to D_{1}(B,A) \to D_{1}(A)\to D_{1}(A|B)\to D_{0}(B,A)\to \cdots .
$$

We have $D_{0}(B,A)=\Omega^{1}_{B} \otimes_{B}A,$ and since $B/\mathbb{Q}$ is smooth,  $D_{1}(B,A)=0.$  In order to compute,  $D_{1}(A|B),$ we use the presentation 
$$
0\to J \to B \to A\to 0.
$$
The naive cotangent complex of $A$ relative to $B$ is then  given by 
$$
J/J^2 \to \Omega^{1} _{B/B}=0.
$$
Since the naive cotangent complex is the good truncation of the cotangent complex in degree 1, it can be used to compute the first Andr\'{e}-Quillen homology. Therefore,  
$D_{1}(A|B)=J/J^2.$ Since the map to  $D_{0}(B,A)=\Omega^{1}_{B}/J\cdot \Omega^{1}_{B}$ is the natural map induced by the differentiation, by the above exact sequence we obtain that 
$$
D_{1}(A)=ker(J/J^2 \xrightarrow{d} \Omega^{1}_{B}/J\cdot \Omega^{1}_{B} ).
$$

\subsection{Branch of the dilogarithm.} Let $A$ be an $k$-algebra, with a square-zero ideal $I$ such that $\underline{A}:=A/I$ is smooth over $k.$  Suppose further that   $\tau: \underline{A} \to A$ is a splitting  of the canonical  projection $A\to \underline{A}.$ The main theorem of this section will be the construction of the regulator map in this context.

\begin{theorem}
Associated to the splitting $\tau,$ there is a regulator map 
$$
\ell i_{2,\tau}: B_{2}(A) \to D_{1}(A),
$$
from the Bloch group to the first Andr\'{e}-Quillen homology of A.
\end{theorem}

We will give two different constructions of $\ell i_{2,\tau}. $ The first one is in terms of explicit formulas. The second one is more conceptual.  

\subsubsection{First construction}\label{First construction} 

We will first construct $\ell i_{2,\tau}$ using various choices and then show that the construction is independent of these choices. 
 Using the splitting $\tau$ we regard $A $  as an $\underline{A}$-algebra.  Express $A$  as a quotient $B \twoheadrightarrow A$ of a  smooth $\underline{A}$-algebra $B.$ Let $\hat{B}$ be the completion of $B$ along the kernel of this map,  $\hat{J}$ be  the kernel of the projection    $\hat{B} \twoheadrightarrow A,$  and $\hat{I}$ be the  inverse image of $I$ in $\hat{B.}$  We  denote the structure map $\underline{A} \to \hat{B}$ by $\hat{\tau}.$ Since,  by assumption, $I^2=0,$ we have $\hat{I}^2 \subseteq \hat{J}.$

We define a map 
$$
\ell i_{2,\tau}( \hat{B},\hat{\tau}): \mathbb{Q}[A^{\flat}] \to \ker (\hat{J}/\hat{J}^2 \to \Omega^{1} _{\hat{B}}/\hat{J}\Omega^{1}_{\hat{B}} =\Omega^{1} _{\hat{B}/\hat{J}^2}/\hat{J}\Omega^{1}_{\hat{B}/\hat{J}^2} ) , 
$$
by sending $[a ]$ to 
$$
-\frac{1}{2}\frac{(\tilde{a}-\hat{\underline{a}})^3}{\hat{\underline{a}}^2(\hat{\underline{a}}-1)^2},
$$
where $\hat{\underline{a}}:=\hat{\tau}(\underline{a}),$ with $\underline{a}$ is the image of $a$ under the map $A \twoheadrightarrow \underline{A},$ and $\tilde{a}$ is any lifting  of $a \in A$ to an element in $\hat{B}.$ 
Note that the value of $\ell i_{2,\tau}( \hat{B},\hat{\tau})$ at $[a]$  is 0, if $a \in \tau(\underline{A}).$ 

First, we show that the value of $\ell i_{2,\tau}( \hat{B},\hat{\tau})$ on $[a]$ does not depend on the  choice of the lifting $
\tilde{a} $ in $\hat{B}$ and that it lands in the above subspace of $\hat{J}/\hat{J}^2.$  

\begin{proposition}\label{well-def} For $a \in A^{\flat},$ 
 $\ell i_{2,\tau} (\hat{B},\hat{\tau})([ a ])$ is a  well-defined element of $\ker (\hat{J}/\hat{J}^2 \to \Omega^{1} _{\hat{B}}/\hat{J}\Omega^{1}_{\hat{B}}).$
\end{proposition}

\begin{proof}
Using the splittings $\tau$ and $\hat{\tau},$ we will assume without loss of generality that 
$A=\underline{A} \oplus I$ and $\hat{B}=\underline{A} \oplus \hat{J} .$  

First, let us show that the definition is independent of the choice of the lifting $\tilde{a}.$ If $\tilde{a}'$ is another lifting, then $\tilde{a} '=\tilde {a} + \alpha, $ for some $\alpha \in \hat{J}.$   We need to show that 
$$
(\tilde{a}'-\underline{\hat{a}})^3-(\tilde{a}-\underline{\hat{a}})^3 =0
$$
 in $\hat{B}/\hat{J}^2.$ 
Letting $\tilde{b}:=\tilde{a} -\hat{\underline{a}} \in \hat{I},$ this is equivalent to showing that 
$$
(\tilde{b}+\alpha)^3- \tilde{b}^3=0 
 $$
 in $\hat{B}/\hat{J}^2.$ 
 Since $\alpha \in \hat{J},$ $\tilde{b}\in \hat{I}$ and $\hat{I}^2 \subseteq \hat{J},$ the above expression is in   $ (\hat{I}^2 \hat{J}+\hat{J}^2) \subseteq  \hat{J}^2.$ This proves  the independence with respect to the choice of the lifting $\tilde{a}.$ 
 
 Since $\tilde{a}-\hat{\underline{a}} \in \hat{I}$ and $\hat{I}^2 \subseteq \hat{J},$ $\ell i_{2,\tau} (\hat{B},\hat{\tau})(\{ a ) \in \hat{I}^3/\hat{J}^2 \subseteq \hat{J}/\hat{J}^2.$ Therefore, the  image of this element under $d$  lies in  $d(\hat{I}^3) \subseteq \hat{I}^2 \Omega_{\hat{B}/\hat{J}^2} \subseteq \hat{J} \Omega^{1} _{\hat{B}/\hat{J}^2}.$   This finishes the proof of the proposition. 
 \end{proof}
We will abuse the notation and denote $\ell i_{2,\tau} (\hat{B},\hat{\tau})([a])$ by $\ell i_{2,\tau} (\hat{B},\hat{\tau})(a)$. Next we prove that $\ell i_{2,\tau} (\hat{B},\hat{\tau})$  satisfies the five-term functional equation of the dilogarithm and hence descends to give a map from $B_{2}(A).$ 

\begin{proposition}
 The above function $\ell i_{2,\tau} (\hat{B},\hat{\tau})$ factors through the projection 
 $
 \mathbb{Q}[A^{\flat}] \twoheadrightarrow B_{2}(A)
 $
 to induce a map 
 $$
 B_{2}(A) \to \ker (\hat{J}/\hat{J}^2 \to \Omega^{1} _{\hat{B}}/\hat{J}\Omega^{1}_{\hat{B}}).
 $$
\end{proposition}

\begin{proof}
Again, without loss of generality, we  assume that $A$ and $\hat{B}$ are split as above. For $x:=a+\alpha \in A=\underline{A} \oplus I$ with $a\in \underline{A} ^{\flat}$ and $\alpha \in I , $ $\ell i_{2,\tau}$ maps $x$ to 
$$
-\frac{1}{2} \frac{\tilde{\alpha}^3}{a^2 (a-1)^2},
$$
where $\tilde{\alpha}$ is any lifting of $\alpha$ to an element of $\hat{J} \subseteq \hat{B}=\underline{A}\oplus \hat{J} .$  If $y:=b+\beta$ is a similar element, we need to show that $\ell i_{2,\tau}$ maps 
$$
[x] -[y]+[\frac{y}{x}] -[\frac{1-x^{-1}}{1-y^{-1}} ] +[\frac{1-x}{1-y} ]
$$
to 0. Fix liftings  $\tilde{\alpha}$ and $\tilde{\beta}$ of $\alpha $ and $\beta$ to $\hat{J}.$ 
Note that $\frac{y}{x}=\frac{b}{a}+\frac{b}{a}(\frac{\beta}{b}-\frac{\alpha}{a}).$ Using the lifting $\frac{b}{a}(\frac{\tilde{\beta}}{b}-\frac{\tilde{\alpha}}{a})$ of $\frac{b}{a}(\frac{\beta}{b}-\frac{\alpha}{a}),$ we see that  $\ell i_{2,\tau}$ maps $\frac{y}{x}$ to 
$$
-\frac{1}{2}\frac{(a\tilde{\beta}-b\tilde{\alpha})^3}{(ab(a-b))^2}.
$$
Similarly, $\frac{1-x^{-1}}{1-y^{-1}}=\frac{1-a^{-1}}{1-b^{-1}}+\frac{1-a^{-1}}{1-b^{-1}}(\frac{a^{-2}\alpha}{1-a^{-1}}-\frac{b^{-2}\beta}{1-b^{-1}})$ and using the lifting $\frac{1-a^{-1}}{1-b^{-1}}(\frac{a^{-2}\tilde{\alpha}}{1-a^{-1}}-\frac{b^{-2}\tilde{\beta}}{1-b^{-1}}),$ we see that this is mapped to 
$$
-\frac{1}{2} \frac{(b(b-1)\tilde{\alpha}-a(a-1)\tilde{\beta})^3}{(ab(a-1)(b-1)(a-b))^2}.
 $$
 Finally, $\frac{1-x}{1-y}=\frac{1-a}{1-b}+\frac{1-a}{1-b}(\frac{\beta}{1-b}-\frac{\alpha}{1-a}) ,$ and using the lifting $\frac{1-a}{1-b}(\frac{\tilde{\beta}}{1-b}-\frac{\tilde{\alpha}}{1-a}),$ we see that $\ell i_{2,\tau}$ maps $\frac{1-x}{1-y}$ to 
 $$
 -\frac{1}{2}\frac{((b-1)\tilde{\alpha}-(a-1)\tilde{\beta})^3}{((a-1)(b-1)(a-b))^2}. 
 $$
 The functional equation is then a consequence of the following  identity: 
 $$
 \frac{\tilde{\alpha}^3}{(a(a-1))^2}- \frac{\tilde{\beta}^3}{(b(b-1))^2}+\frac{(a\tilde{\beta}-b\tilde{\alpha})^3}{(ab(a-b))^2}-\frac{(b(b-1)\tilde{\alpha}-a(a-1)\tilde{\beta})^3}{(ab(a-1)(b-1)(a-b))^2}+\frac{((b-1)\tilde{\alpha}-(a-1)\tilde{\beta})^3}{((a-1)(b-1)(a-b))^2}=0.
 $$
 \end{proof}
 
 Finally, we need to show how to get a map to   Andr\'{e}-Quillen homology. Note that $J/J^2=\hat{J}/\hat{J}^2$ and $\Omega^1 _{B}/J \Omega^1 _{B}=\Omega^1 _{\hat{B}}/\hat{J} \Omega^1 _{\hat{B}}. $ Therefore, we have 
 $$
 \ker (J/J^2 \to \Omega^{1} _{B}/J\Omega^{1}_{B} )= \ker (\hat{J}/\hat{J}^2 \to \Omega^{1} _{\hat{B}}/\hat{J}\Omega^{1}_{\hat{B}} ) .
 $$
 Since $B$ is a smooth $\underline{A}$-algebra and $\underline{A}$ is smooth over $\mathbb{Q}$, $B$ is smooth over $\mathbb{Q}.$
 By the discussion in \textsection \ref{AQ}, we therefore have $ \ker (J/J^2 \to \Omega^{1} _{B}/J\Omega^{1}_{B} )=D_{1}(A).$     We thus obtain the  map from $B_{2}(A)$ to $D_{1}(A)$ we were looking for.  We need to show that this map is independent of the choice of the presentation $B \twoheadrightarrow A,$ where $B/\underline{A}$ is smooth. If $B' \twoheadrightarrow A$ is another such map by considering $B \otimes_{\underline{A}} B' \twoheadrightarrow A,$ we may assume without loss of generality that there is a commutative triangle 
 $$
 \xymatrix{
 B'\ar[d] \ar@{->>}[rd] \\
 B\ar@{->>}[r] & A 
 .}
 $$
 Using the functoriality of the construction above, we achieve the  proof of the independence.

\subsubsection{Second construction.} In this subsection, we give a more conceptual construction of the above regulator. As above,  we start with a presentation $B\twoheadrightarrow A$ of $A$ as a quotient of a smooth $\underline{A}$-algebra $B$  with kernel $J.$ 

 Let $d:\hat{B} \to \Omega^1 _{\hat{B}}$ denote the absolute differential and $\underline{d}: \hat{B} \to \Omega^1 _{\hat{B}/\underline{A}}$ the differential relative to $\underline{A}.$ Note that since $\hat{B}$ is considered as an $\underline{A}$-algebra via the structure map $\hat{\tau},$ the map $\underline{d}$ depends on $\hat{\tau}$ even though we suppress this in the notation.   

We start with defining a map 
$$
\xymatrix{
\Lambda^2 \hat{B}^{\times} \ar^{-3 \cdot \log ^{\circ} _{\hat{\tau}} \wedge \underline{d} log }[rr] & & \Omega^{1} _{\hat{B}/\underline{A}}
},
$$
which we will integrate on the image of $\delta$ in a suitable sense. The map $\log ^{\circ} _{\hat{\tau}}:\hat{B}^{\times} \to \hat{B}$ is defined by $\log ^{\circ} _{\hat{\tau}}(a):= \log (\frac{a}{\hat{\tau}(\underline{a})})$ and $\underline{d} log: \hat{B}^{\times}  \to \Omega^{1} _{\hat{B}/\underline{A}}$ by $\underline{d} log(a):=\frac{\underline{d}(a)}{a}.$ Then $\log ^{\circ} _{\hat{\tau}} \wedge \underline{d} log $ is the map sending $a\wedge b$ to $\log ^{\circ} _{\hat{\tau}}(a) \underline{d} log(b)-\log ^{\circ} _{\hat{\tau}}(b) \underline{d} log(a).$

We will also need the following basic lemma which implies the uniqueness of the anti-derivative if we restrict to the infinitesimal part. 

\begin{lemma}
The map 
$\hat{B}^{\circ} \xrightarrow{\underline{d}} \Omega^1 _{\hat{B}/\underline{A}}$
is injective. 
\end{lemma}

\begin{proof}
We may assume without loss of generality that the spectra of $\underline{A}$ and  $\hat{B}$ are connected. With this assumption, $\underline{A}$ and    $\hat{B}$ are an integral domains. Fix a point  $x$ of $\underline{A},$ and let  $\underline{A}_{\hat{x}}$ and $\hat{B}_{\hat{x}}$ be the completions of $\underline{A}$ and $\hat{B}$ at $x.$ The map $\hat{B} \to \hat{B}_{\hat{x}}$ is then an injection. By the smoothness assumptions, $\underline{A}_{\hat{x}}\simeq K[[x_1,\cdots,x_{n}]]$ and $\hat{B}_{\hat{x}}\simeq K[[x_1,\cdots,x_{n},y_1,\cdots,y_m]].$ This implies that the kernel of the map $\underline{d}: \hat{B}_{\hat{x}} \to \Omega^{1} _{\hat{B}_{\hat{x}}/\underline{A}_{\hat{x}}}$ is $\underline{A}_{\hat{x}},$ and hence the map is injective on $\hat{B}_{\hat{x}}^{\circ}.$ Combined with the above injectivity of the completion map finishes the proof of the lemma. 
\end{proof}

The regulator we are looking for is the map from $B_{2}(A)$ to $\ker(d)$ that is induced from the diagram below.

$$
\xymatrix{ 
\mathbb{Q}[\hat{B}^{\flat}] \ar@{->>}[d] \ar@{-->}[rrddd]  \ar@{->>}[rrr]& & & B_{2}(\hat{B}) \ar@{-->}[dddl] \ar@{->>}[dll] \ar@{-->}[ldd]  \ar@{-->}[d] \ar^{\delta}[r] & \Lambda^2 \hat{B}^{\times}   \ar^{-3 \cdot \log ^{\circ} _{\hat{\tau}} \wedge \underline{d} log }[d]   \\
\mathbb{Q}[A^{\flat}] \ar@{-->}[rrdd] \ar@{->>}[r] & B_2(A)   \ar@{-->}[rdd] & & \hat{J} \ar@{^{(}->}^{\underline{d}}[r] \ar@{->>}[d] & \Omega^1 _{\hat{B}/\underline{A}} \ar@{->>}[d] \\
& & \ker(\underline{d}) \ar@{^{(}->}[r] &\hat{J} /\hat{J}^2 \ar^{\underline{d}}[r] &  (\Omega^1 _{\hat{B}/\underline{A}}/\hat{J})^{\circ} \\
& & \ker(d) \ar@{^{(}->}[r] \ar@{^{(}->}[u] &\hat{J} /\hat{J}^2 \ar^{d}[r] \ar[u]^*[@]{\sim} &  (\Omega^1 _{\hat{B}}/\hat{J})^{\circ} \ar@{->>}[u]
}
$$
We explain this in detail below. The map $\delta$ sends $[a+\tilde{\alpha}] , $ with $a \in \underline{A}$ and $\tilde{\alpha} \in \hat{I},$  to  $(a-1) (1+\frac{\tilde{\alpha}}{a-1}) \wedge a (1+\frac{\tilde{\alpha}}{a}) \in \Lambda ^{2} \hat{B} ^{\times}.$ The image of this element under   
 $\log^{\circ} _{\hat{\tau}}\wedge \underline{d}log $ is 
\begin{eqnarray*}
\frac{1}{2}\frac{\tilde{\alpha}^2 d \tilde{\alpha}}{(a-1)^2a^2}+O(\tilde{\alpha}^3)d \tilde{\alpha}.
\end{eqnarray*}
There is a unique element $$-\frac{1}{2}\frac{\tilde{\alpha}^3}{(a(a-1))^2}+O(\tilde{\alpha}^4) \in \hat{J}$$ whose image under $\underline{d}$ is exactly
$$
-3\cdot (\log ^{\circ} _{\hat{\tau}} \wedge \underline{d} log) (\delta(\{a+\tilde{\alpha} )).
$$
This element maps to $-\frac{1}{2}\frac{\tilde{\alpha}^3}{(a(a-1))^2} \in \hat{J}/\hat{J}^2,$ which lies in $\ker(d)$ since $\tilde{\alpha}^2 \in \hat{I}^2 \subseteq \hat{J}.$ This defines a map 
$$
B_{2}(\hat{B}) \to \ker(d).
$$ 
By the definition of the Bloch group,  the diagram 
$$
\xymatrix{
\mathbb{Q}[\hat{B}^{\flat}] \ar@{->>}[d]   \ar@{->>}[r]&  B_{2}(\hat{B}) \ar@{->>}[d] \\
\mathbb{Q}[A^{\flat}] \ar@{->>}[r]& B_{2}(A) 
}
$$
is co-cartesian. Therefore, in order to prove that the map from $B_{2}(\hat{B})$ to $\ker(d)$ factors through the projection $B_{2}(\hat{B}) \to B_{2}(A),$ it is necessary and sufficient to prove that its composition with $\mathbb{Q}[\hat{B}^{\flat}] \to B_{2}(\hat{B})$ factors through the projection $\mathbb{Q}[\hat{B}^{\flat}] \to \mathbb{Q}[A^{\flat}].$ By the formula for the map above, this boils down to showing that if $\alpha \in I$ and $\tilde{\alpha}$ and $\tilde{\beta}$ are two different liftings to $\hat{I},$ then the reductions of $\tilde{\alpha} ^{3}$ and $\tilde{\beta} ^{3}$ in $\hat{J}/\hat{J}^2$  are the same. This follows from $\tilde{\beta} ^{3}-\tilde{\alpha}^{3}=(\tilde{\beta} -\tilde{\alpha})(\tilde{\beta}^2+\tilde{\beta}\tilde{\alpha}+\tilde{\alpha}^2) \in \hat{J}\cdot \hat{I}^2 \subseteq \hat{J}^2.$

In order to show that the map is independent of the presentation, and thus defines a map to $D_{1}(A)$, we can argue exactly as in the last paragraph of \textsection \ref{First construction}.

\section{Map from cyclic homology to the Bloch group}\label{loc-bloch} 
\subsection{Comparison of cyclic homology and Andr\'{e}-Quillen homology}  Suppose that $A$ is as in the previous section. We will, in fact, assume that  $A=\underline{A}\oplus I,$ using the given splitting $\tau.$  First we recall that cyclic homology can be computed in terms of Andr\'{e}-Quillen homology. 

\begin{lemma}
With $A=\underline{A}\oplus I$ as above with $\underline{A}$ smooth over $k$, we have 
$$
HC_{2} ^{\circ}(A)^{(1)}=D_{1}(A). 
$$
\end{lemma}

\begin{proof}
Since Connes' long exact sequence degenerates for graded algebras, we have the exact sequence \cite[Theorem 4.1.13]{lod}
$$
0\to HC_{n-1} ^{\circ}(A)\to HH_{n} ^{\circ}(A) \to HC_{n} ^{\circ}(A)\to 0.
$$
Using the fact that this exact sequence is compatible with  $\lambda$-decomposition \cite[Theorem 4.6.9]{lod} and the fact that $HC_{1} ^{\circ}(A)^{(0 )}=0,$  we deduce the isomorphism
$
 HH_{2} ^{\circ}(A)^{(1)} \xrightarrow{\sim} HC_{2} ^{\circ}(A) ^{(1)}.
$
Finally, by \cite[Proposition 4.5.13]{lod}, this part of  Hochschild homology coincides with the Andr\'{e}-Quillen homology: 
$$
 HH_{2} ^{\circ}(A)^{(1)}\simeq D_{1} ^{\circ}(A).
$$
By the smoothness assumption $D_{1}(\underline{A})=0$ and we have the statement. 
\end{proof}

\subsection{Map from cyclic homology to the Bloch group.} In order to prove the injectivity of the regulator,  we will have to first recall the natural map from cyclic homology to the Bloch group. 

\subsubsection{Review of the map from cyclic homology.} 

By \cite{unv1} we have a map, 
$$
HC_{2} ^{\circ}(A)
^{(1)} \to B_{2} ^{\circ}(A).
$$
Composing with $\ell i_{2},$ the branch of the dilogarithm corresponding to the above splitting of $A,$ we obtain a map $HC_{2} ^{\circ}(A)^{(1)} \to D_{1}(A).$ First we would like to give an example which shows that this map is not  always an isomorphism. In fact, in later sections this example will serve as a counterexample to various other statements such as  the generalization of the homotopy map. Because of this example, we will have to add the  assumption that $I$ is a free $\underline{A}$-module in order to prove the injectivity of the regulator.  

\begin{example}\label{example}
By the formula for $\ell i_{2},$ for $a \in \underline{A},$ $\alpha \in A$ and $\lambda \in \mathbb{Q},$ we have $\ell i_{2}(a+\lambda \alpha)=\lambda ^3 \ell i_{2}(a+\alpha). $ Therefore, if $t_{\lambda}:A \to A$ denotes the ring homomorphism defined by $t_{\lambda}(a+\alpha):=a+\lambda \alpha,$ and $t_{\lambda,*}$ is the induced map on $B_{2}(A),$ then $\ell i_{2}(t_{\lambda,*}(q))=\lambda^3\ell i_{2}(q),$ for any $q\in B_{2}(A).$

Let $\underline{A}=\mathbb{Q}[x],$  and $A:=\underline{A}[t]/(xt,t^2)$ denote the square-zero infinitesimal thickening of $\underline{A}.$ Letting $B:=\underline{A}[t]$ and $J:=(xt,t^2),$ we have a presentation $0\to J\to B\to A \to 0,$ which can be used to compute $D_{1}(A)$ as
$$
D_{1}(A)=\ker(J/J^2 \to \Omega^1 _{B}/J).
$$
The element $xt^2$  in $\ker(J/J^2 \to \Omega^{1} _{B}/J)=D_{1}(A)$ is non-zero and corresponds to a multiple of $t\otimes x \otimes t$ in $HC_{2}^{\circ}(A)^{(1)}.$ In particular, this element in cyclic homology is nonzero. Let $q$ be its image in $B_{2}(A).$ Note that $t_{\lambda,*}(q)=\lambda^2q.$ Comparing with the above formula, we obtain that $\lambda ^3 \ell i_{2}(q)=\lambda ^2 \ell i_{2}(q).$ Therefore, $\ell i_{2}(q)=0$ and $\ell i_{2}$ is not injective on $HC_{2}^{\circ} (A) ^{(1)}.$ 

\end{example}

\subsubsection{Injectivity of the regulator on cyclic homology}\label{smooth over artin}

Suppose further in this section that $I$ is a free $\underline{A}$-module. Note that since $I^2=0,$ this is equivalent to assuming that the conormal module associated to the projection $A \to \underline{A}$ be a free $\underline{A}$-module. Using a splitting, without loss of generality, we are in the above situation $A=\underline{A}\oplus I,$ with $I$ a {\it free} $\underline{A}$-module. 

Let $S^{ \bigcdot}(I)$ denote the symmetric algebra of $I,$ $\tilde{I}$ the augmentation ideal and $J:=\tilde{I}^2.$  Since $S^{ \bigcdot}(I)$ is smooth over $\underline{A}$ and the sequence $0 \to J \to S^{ \bigcdot}(I) \to A\to 0$ is a presentation  of $A,$  $D_{1}(A)$ is given by 
$$
\ker(J/J^2\to \Omega^{1} _{S^{ \bigcdot}(I)}/J)=\tilde{I}^{3}/\tilde{I}^4\xleftarrow{\sim} S^3(I).
 $$
 Therefore,  we have an isomorphism $HC_{2} ^{\circ}(A)^{(1)} \simeq  S^3(I).$ 
 
 On the other hand, there is another natural map defined as follows. Let $\sigma$ be the 3-cycle $(123),$ and $C_{3}:=\langle \sigma \rangle $  the subgroup of $S_{3}$ generated by $\sigma.$ $C_{3}$ naturally acts on $I^{\otimes 3},$ with the  rule that  $\sigma$ sends $a\otimes b \otimes c$ to   $c \otimes a\otimes b.$ Let $I^{\otimes 3}/(1-\sigma)$ denote the group of co-invariants with respect to this action. The natural map $I^{\otimes 3}/(1-\sigma)\to HC_{2} ^{\circ}(A)^{(1)}$ factors via the projection $I^{\otimes 3}/(1-\sigma) \twoheadrightarrow I^{\otimes _{\underline{A}}3}/(1-\sigma).$ Therefore, we obtain a map 
 $$
 I^{\otimes _{\underline{A}}3}/(1-\sigma)\to HC_{2} ^{\circ}(A)^{(1)}.
 $$
Composing with the map 
$$
S^{3}(I)\hookrightarrow I^{\otimes_{\underline{A}} 3}  /(1-\sigma)
$$
which sends $a\otimes b \otimes c$ to $\frac{1}{2} (a\otimes b \otimes c+a\otimes c \otimes b),$ we obtain a map 
$$
S^{3}(I) \to HC_{2} ^{\circ}(A)^{(1)}.
$$
This map, being a non-zero constant multiple of the above map, is an isomorphism. 

Let us now look at the composition 
$$
S^{3}(I) \xrightarrow{\sim} HC_{2} ^{\circ}(A)^{(1)}\to B_{2}(A)\to D_{1}(A)\simeq S^{3}(I).
$$

\begin{proposition}\label{reg-on-cyclic}
The above composition $S^{3}(I) \to S^{3}(I)$ is multiplication by 3. 
\end{proposition}

\begin{proof}
 The arguments in this proof are  analogous to those of \cite[\textsection 4]{unv1}. We need to compute the image of the element $a\otimes b \otimes c$ in $I^{\otimes 3}$ under the above composition. This element is mapped to the 3-cycle $\beta_{abc}:=ae _{12}\wedge b e_{23} \wedge c e_{31}$ in the homology of $\mathfrak{gl}_{3}(A).$ In $C_{*}(\mathfrak{gl}_3(A)),$ we have 
$$
d(e_{13} \wedge ae_{12}\wedge be_{21}\wedge ce_{31})=-\beta_{abc}+\gamma_{abc}-ae _{12}\wedge b e_{21} \wedge c e_{33},
$$
where $\gamma_{abc}:=ae _{12}\wedge b e_{21} \wedge c e_{11}.$ We claim that the image of the term $ae _{12}\wedge b e_{21} \wedge c e_{33}$ is equal to 0.
Since 
$$
d(e_{12} \wedge ae_{11}\wedge be_{21}\wedge ce_{33})=-ae _{12}\wedge b e_{21} \wedge c e_{33}+ae _{11}\wedge b e_{11} \wedge c e_{33}-ae _{11}\wedge b e_{22} \wedge c e_{33},
$$
it suffices to show that the images of $ae _{11}\wedge b e_{11} \wedge c e_{33}$ and $ae _{11}\wedge b e_{22} \wedge c e_{33}$ are 0. In order to see this, note that  these elements can be lifted to cycles $\tilde{a}e _{11}\wedge \tilde{b} e_{11} \wedge \tilde{c} e_{33}$ and $\tilde{a}e _{11}\wedge \tilde{b} e_{22} \wedge \tilde{c} e_{33}$ in 
$C_{*}(\mathfrak{gl}_3(S^{\bigcdot}(I))),$ which immediately implies that their images in $S^{3}(I)$ are 0 as in the proof of \cite[Lemma 4.2.1]{unv1}. Therefore, the image of  $\beta_{abc}$ coincides with that of $\gamma_{abc}.$ 

Note that $\gamma_{abc}$ is a 3-cycle in $C_{*}(\mathfrak{gl}_2(A).$ In order to compute its image, we start with a lifting of this element to an element in $C_{*}(\mathfrak{gl}_2(S^{\bigcdot}(I))).$ Let $\tilde{\gamma}_{abc}:=\tilde{a}e_{12}\wedge \tilde{b} e_{21} \wedge \tilde{c} e_{11}$ be such a lifting. Its boundary is equal to 
\begin{eqnarray}\label{elt}
d(\tilde{\gamma}_{abc})=\tilde{a}\tilde{b}e_{11} \wedge \tilde{c}e_{11} -\tilde{a}\tilde{b}e_{22} \wedge \tilde{c}e_{11}+\tilde{a}\tilde{c}e_{12} \wedge \tilde{b}e_{21}-\tilde{a}e_{12} \wedge \tilde{b}\tilde{c}e_{21}.
\end{eqnarray}
By the choice of the vectors $v_1=(1,1), \, v_2=(0,1)$ and $v_3=(1,0),$ we deduce as in \cite[Lemma 4.2.4]{unv1} that the image of $\tilde{a}\tilde{b}e_{11} \wedge \tilde{c}e_{11} $ in $\Lambda^2 \hat{S}^{\bigcdot}(I) ^{\times}$ is equal to 0. Note that after obtaining the image of $d(\tilde{\gamma}_{abc})$ in $\Lambda^2 \hat{S}^{\bigcdot}(I) ^{\times},$ we apply $-3 \log ^{\circ} \wedge \underline{d} log $ to get an element in $\Omega^{1}_{\hat{S}^{\bigcdot}(I)/\underline{A}}.$ 

In order to compute the image of $\gamma_{abc}$ in  $S^{3}(I)=\tilde{I}^3/\tilde{I}^4,$
we need to find the element in  $\hat{S}^{\bigcdot}(I)$ which maps to the image of this element  in $\Omega^{1}_{\hat{S}^{\bigcdot}(I)/\underline{A}}$ and then reduce it modulo $\tilde{I}^4.$

First let us study the images of terms of the type $\alpha e_{11}\wedge \beta e_{22},$ with $\alpha \in \tilde{I}$ and $\beta \in \tilde{I}^2.$ This is mapped to the element $$3 (\log ^{\circ} \wedge \underline{d} log)(\exp(\alpha)\wedge \exp (\beta))+\tilde{I}^3\underline{d}(\tilde{I})=3(\alpha \underline{d}\beta-\beta \underline{d} \alpha)+\tilde{I}^3\underline{d}(\tilde{I}),$$
in $\Omega^{1} _{\hat{S}^{\bigcdot}(I)/\underline{A}}/\tilde{I}^3\underline{d}(\tilde{I}),$ by the same computation in \cite[Lemma 4.2.6]{unv1}. 

With the same $\alpha$ and $\beta,$ $\alpha e_{12}\wedge \beta e_{21}$ is mapped to 
$$-3 (\log ^{\circ} \wedge \underline{d} log)(\exp(-\alpha)\wedge \exp (-\beta))+\tilde{I}^3\underline{d}(\tilde{I})=-3(\alpha \underline{d}\beta-\beta \underline{d} \alpha)+\tilde{I}^3\underline{d}(\tilde{I})$$ 
in $\Omega^{1} _{\hat{S}^{\bigcdot}(I)/\underline{A}}/\tilde{I}^3\underline{d}(\tilde{I}),$
by  the same computation as in \cite[Lemma 4.2.7]{unv1}.

Applying the above formulas to the expression (\ref{elt}) of  $d(\tilde{\gamma}_{abc}),$ we deduce that the image of $d(\tilde{\gamma}_{abc})$ in  $\Omega^{1} _{\hat{S}^{\bigcdot}(I)/\underline{A}}/\tilde{I}^3\underline{d}(\tilde{I})$ is $-3$ times 
$$
-\tilde{c}\cdot \underline{d}(\tilde{a}\tilde{b})+\tilde{a}\tilde{b}\cdot \underline{d}(\tilde{c})+\tilde{a}\tilde{c}\cdot \underline{d}(\tilde{b})-\tilde{b}\cdot \underline{d}(\tilde{a}\tilde{c})-\tilde{a}\cdot \underline{d}(\tilde{b}\tilde{c})+\tilde{b}\tilde{c}\cdot \underline{d}(\tilde{a})+\tilde{I}^3\underline{d}(\tilde{I})=-\underline{d}(\tilde{a}\tilde{b}\tilde{c})+ \tilde{I}^3\underline{d}(\tilde{I}).
$$
By reducing  the anti-derivative of this element modulo $\tilde{I}^4,$ we conclude that the image of $\gamma_{abc}$ in $\tilde{I}^3/\tilde{I}^4$ is $3\tilde{a}\tilde{b} \tilde{c}.$ This finishes the proof of the proposition.
\end{proof}

\begin{corollary}\label{localiso}
Suppose that $A$ is a local  $k$-algebra with a square-zero ideal $I$ such that $I$ is free over $\underline{A}=A/I$ and $\underline{A}/k$ is smooth.  If $\tau:\underline{A} \to A$ is a splitting of the canonical  projection, then the composition 
$$
HC_{2}^\circ(A) ^{(1)}\to B_{2} ^{\circ}(A) \xrightarrow{\ell i_{2,\tau}}  D_{1}(A)  
$$
is an isomorphism. 
\end{corollary}

\begin{proof}
Using the splitting $\tau,$ we reduce to the case when $A=\underline{A}\oplus I,$ with $I$ a free $\underline{A}$-module. The statement is then a consequence of the above proposition.  
\end{proof}

\section{Homotopy map} \label{homotopy section}

In this section, we would like to compare the different branches of the dilogarithm corresponding to different splittings. First, we give an example that this is not possible for all square-zero extensions. 

\begin{example}\label{example2}
Assume the same notation as in Example \ref{example}. Let $\tau$ be the  splitting of $A$ corresponding to the direct sum decomposition $A=\underline{A}\oplus I,$ where $I=(t),$ and $\sigma:\underline{A} \to A$ the splitting such that $\sigma(x)=x+t.$ Let $q$ be the image of the element $t\otimes x \otimes t \in HC_{2} ^{\circ}(A)^{(1)}$ in $B_{2} ^{\circ} (A).$ We have seen in Example \ref{example} that $\ell i_{2,\tau}(q)=0.$ In order to compute $\ell i_{2,\sigma}(q),$ by an elementary transport of structure, we see that  $\ell i_{2,\sigma}(q)=\ell i_{2,\tau}(q'),$ where $q'$ is the image of $t \otimes (x-t)\otimes t \in HC_{2} ^{\circ} (A)^{(1)}$ in $B_{2} ^{\circ}(A).$ If $q''$ is the image of $t\otimes t \otimes t \in HC_{2} ^{\circ} (A) ^{(1)}$  in $B_{2}^{\circ} (A),$ then $q'=q-q''.$  By Proposition \ref{reg-on-cyclic},  we have  $\ell i_{2,\tau}(q')=\ell i_{2,\tau}(q)-\ell i_{2,\tau}(q'')=-\ell i_{2,\tau}(q'')=-3t^3 \in (t^3)/(t^4).$ In particular, since the last element is non-zero, $\ell i_{2,\tau}(q)\neq \ell i_{2,\sigma}(q).$ Note that that $q \in \ker (\delta),$ but the different branches of the dilogarithm do {\it not} necessarily have the same value on $q.$ We will see below that this cannot happen if $I$ is a free $\underline{A}$-module  as in \textsection \ref{smooth over artin}.
\end{example}

Let $A_i,$ for $i=1,\,2,$ be  $k$-algebras, with  square-zero ideals $I_i,$ as in \textsection \ref{smooth over artin}. Suppose that  $f:A_1 \to A_2$  is a $k$-algebra homomorphism and that 
$$
\tau_1: \underline{A}_1 \to A_1
$$
and 
$$
\tau_2: \underline{A}_2 \to A_2,
$$
are splittings of $k$-algebras. We do {\it not} require that $f$ be compatible with the $\tau_i$'s. Namely, if $\underline{f}: \underline{A}_1 \to \underline{A}_{2}$ denotes the map induced by  $f,$  then $f \circ \tau_1$ is not necessarily equal to $\tau_2  \circ\underline{f}.$

 We would like to define a map 
$$
h_f(\tau_1, \tau_2)
: F((\Lambda^{2}A_1^{\times})^{\circ}) \to D_1(A_2), 
$$
with the property that, for every $\alpha \in B_{2} ^{\circ}(A_1),$ we have 
\begin{align}\label{eqhom}
\ell i_{2,\tau_2 } (f(\alpha))-f_{*}(\ell i_{2,\tau_1 } (\alpha))=h_f(\tau_1,\tau_2)(\delta(\alpha)).
\end{align}
This proved in Propostion \ref{homprop} below. This map is a measure of  the defect between $f \circ \tau_1$ and $\tau_2  \circ\underline{f}.$ In the sense that, $h_f(\tau_1, \tau_2)$ is 0 if $f \circ \tau_1=\tau_2  \circ\underline{f}.$ An explicit formula for this map is given in Proposition \ref{equation h}.

\subsection{Construction of $h_f(\tau, \sigma)$}

In this section, we will construct the above mentioned homotopy. Let $B_{i}$ denote the symmetric algebra $S^{\bigcdot} _{\underline{A}_i} (I_{i})$ of the free $\underline{A}_i$-module $I_{i}$ and $\hat{B}_{i}$ its completion along the augmentation ideal. Let $\hat{\tau}_i:\underline{A}_{i} \to \hat{B}_i$ denote the structure map. There is a natural surjection $\hat{B}_{i} \to A_{i},$ with kernel $\hat{J}_i,$ which is the square of the augmentation ideal $\hat{I}_{i}.$

\begin{lemma}
There is a map $\hat{f}:\hat{B}_{1} \to \hat{B}_2$ which makes the diagram: 
$$
\xymatrix{ 
 \hat{B}_1 \ar@{->>}[d] \ar^{\hat{f}}[r] & \hat{B}_2 \ar@{->>}[d] \\
A_1 \ar^{f}[r] & A_{2}
}
$$
commute. We do {\it not} require that $\hat{f}$ be compatible with $\hat{\tau}_i,$ in the sense above.  
\end{lemma}

\begin{proof}
Since we do not require compatibility with the $\tau_{i}$'s, we will assume, without loss of generality, that $A_{i}=\underline{A}_i \oplus I_{i}$ are split. In this setting,  $f$ need {\it not} map $\underline{A}_{1} \subseteq A_{1}$ into $\underline{A}_{2}\subseteq A_2.$ 
Since $\underline{A}_{1}/k$ is smooth we can lift the map $\underline{A}_{1}\to A_{2}=\hat{B}_{2}/\hat{I}_{2} ^2,$ using successive thickenings, to a map $\underline{A}_{1} \to \hat{B}_{2}.$ We can extend this  to a map $B_{1}\to \hat{B}_{2}$ by sending $\alpha \in I_{1} $ to $f(\alpha) \in I_{2} \subseteq \hat{B}_{2}.$ This factors through the completion to give a map as in the statement of the lemma. 
\end{proof}
The map $\hat{f}$ is not unique. We will first define a map 

$$
h_{\hat{f}}(\hat{\tau}_1,\hat{\tau}_2): F((\Lambda^{2}(\hat{B}_1/\hat{J}_{1} ^2)^{\times})^{\circ}) \to \hat{B}_2^{\circ}/\hat{J}^2 _2.
$$
Let  $F((\Lambda^{2}\hat{B}_1^{\times})^{\circ}) :=\ker((\Lambda^{2} \hat{B}_1 ^{\times})^{\circ} \to (\Omega^{1} _{\hat{B}_1}/d\hat{B}_1)^{\circ} ).$ Note that since $\hat{B}_1$ is not an artin ring, $F((\Lambda^{2}\hat{B}_1^{\times})^{\circ})$ need not coincide with the image of $\delta ^{\circ}.$

 We have the following diagram: 
$$
\xymatrix{ 
 & F((\Lambda^{2}\hat{B}_1^{\times})^{\circ}) \ar^{ \log ^{\circ} _{\hat{\tau}_1} \wedge \underline{d}log}[d] \ar@{-->}[ld] \ar@{->>}[rrr] & & & F((\Lambda^{2}(\hat{B}_1/\hat{J}_1 ^2)^{\times})^{\circ}) \ar[lllldd]\\
 \hat{B}_1^{\circ} \ar@{->>}[d] \ar@{^{(}->}^{\underline{d}}[r] & \Omega^{1} _{\hat{B}_1/\underline{A}_1}\\
  \hat{B}_1 ^{\circ}/\hat{J}^2_{1}.
}
$$
Therefore, given $\tilde{\alpha} \in F((\Lambda^{2}(\hat{B}_1/\hat{J}_1 ^2)^{\times})^{\circ}) , $ we can define 
$$
\underline{d}^{-1}( \log ^{\circ} _{\hat{\tau}_1} \wedge \underline{d} log) (\tilde {\alpha}) \in \hat{B}_1 ^{\circ}/\hat{J}^2_{1}.
$$
We have a similar diagram for $\hat{B}_2$ and $\hat{\tau}_2$ and since $\hat{f}(\tilde{\alpha}) \in F((\Lambda^{2}(\hat{B}_2/\hat{J}_{2}^2)^{\times})^{\circ}),$ 
$$
\underline{d}^{-1}( \log ^{\circ} _{\hat{\tau}_2} \wedge \underline{d} log) (\hat{f}(\tilde {\alpha})) \in \hat{B}_2 ^{\circ}/\hat{J}^2_{2}.
$$
Finally, we let $h_{\hat{f}}(\hat{\tau}_1,\hat{\tau}_2)(\tilde{\alpha})$ to be $-3$ times  the element
$$
\underline{d}^{-1}( \log ^{\circ} _{\hat{\tau}_2} \wedge \underline{d} log) (\hat{f}(\tilde {\alpha})) -\hat{f}(\underline{d}^{-1}( \log ^{\circ} _{\hat{\tau}_1} \wedge \underline{d} log) (\tilde {\alpha})) \in \hat{B}_2 ^{\circ}/\hat{J}^2_{2}.
$$ 
The following diagram 
$$
\xymatrix{
B_{2} ^{\circ}(\hat{B}_1/\hat{J}_1 ^2) \ar@{->>}[r] \ar@{->>}[d]& F((\Lambda^2(\hat{B}_{1}/\hat{J}_{1} ^2)^{\times})^{\circ}) \ar@{->>}[d] \ar^{h_{\hat{f}}(\hat{\tau}_1,\hat{\tau}_2)}[rd]\\
B_{2} ^{\circ}(A_1) \ar[rd]|-{\ell i_{2,\tau_2}\circ f_*-f_{*}\circ \ell i_{2,\tau_1}}  \ar@{->>}[r] & F((\Lambda^2A_1^{\times})^{\circ})\ar@{-->}^{h_{f}(\tau_1,\tau_2)}[d]& \hat{B}_2 ^{\circ}/\hat{J}^2_{2}\\
& D_{1}(A_2) \ar@{^{(}->}[ur]
}
$$
then defines $h_{f}(\tau_1,\tau_{2}).$ 
Let us first show the commutativity of the outer pentagon. 

\begin{lemma}
The two maps from $B_{2} ^{\circ}(\hat{B}_1/\hat{J}_1 ^2)$ to $\hat{B}_2 ^{\circ}/\hat{J}^2_{2}$ in the above diagram are the same.  
\end{lemma}
\begin{proof}
The definition of $h_{\hat{f}}(\hat{\tau}_1,\hat{\tau}_2)$ 
\end{proof}

The following proposition then finishes the construction of $h_{f}(\tau_1,\tau_2):$ 

\begin{proposition}\label{homprop}
The map 
$$
\ell i_{2,\tau_2}\circ f_*-f_{*}\circ \ell i_{2,\tau_1}: B_{2} ^{\circ}(A_1) \to D_{1}(A_{2})
$$
factors via $\delta^{\circ}:B_{2} ^{\circ}(A_1) \to F((\Lambda^2A_1^{\times})^{\circ})$ to give $h_{f}(\tau_1,\tau_2).$ 
\end{proposition}

\begin{proof}
We need to show that elements in the kernel of $\delta^{\circ}$ are mapped to 0 by the above map.  We will assume, without loss of generality, that $A_{i}$ are local rings. Then the map $HC_{2} ^{\circ}(A_{1})^{(1)} \to B_{2}^{\circ}(A_{1})$ surjects   onto $\ker(\delta^{\circ})$ \cite{unv1}.  Therefore, the statement reduces to the commutativity of the following diagram:
$$
\xymatrix{
HC_{2} ^{\circ}(A_1)^{(1)} \ar^{\ell i _{2,\tau_1}}[d] \ar[r] &  HC_{2} ^{\circ}(A_2)^{(1)} \ar^{\ell i _{2,\tau_2}}[d]\\
D_{1}(A_{1}) \ar[r] & D_{1}(A_{2}).
}
$$
When the source and target of $\ell i_{2,\tau_i}$ are identified with $S^{3}(I_{i}) ,$ through natural morphisms, this map corresponds to multiplication by 3, by Proposition \ref{reg-on-cyclic}. This proves the commutativity of the diagram and finishes the proof of the lemma. 
\end{proof}

\subsection{Explicit computation of $h_{f}(\tau_1,\tau_2)$}\label{explicit hom} We continue to use   the notation of the previous section. Using the splittings, without loss of generality, we assume that $A_i=\underline{A}_i \oplus I_{i}$ and $B_{i}=S^{\bigcdot}(I_i).$ Recall that $\hat{f}:\hat{B}_{1}\to \hat{B}_{2}$ is a lifting of $f.$ If $a \in \hat{B}_{1}$ is a homogenous element of degree $n,$ then we let $\hat{f}_{i}(a)$ be defined by  $\hat{f}=\sum_{0\leq i} \hat{f}_{i}(a),$ with $\hat{f}_{i}(a)$ homogenous of degree $n+i.$ If $f$ is compatible with the splittings,  we can choose $\hat{f}$ so that it is a morphism of graded algebras and hence $\hat{f}_i=0,$ for $0<i.$ However, we do {\it not} assume that $f$ is compatible with these splittings.   
 \subsubsection{Description of $h_{\hat{f}}(\hat{\tau}_1,\hat{\tau}_2)$}\label{Description of hat h}  We continue with the notation above and explicitly compute $h_{\hat{f}}(\hat{\tau}_1,\hat{\tau}_2)$ on certain specific elements in $F(\Lambda ^2 A_{1} ^{\times})^{\circ}.$

(i) Suppose that we are given $(1+\alpha)\wedge (1+a \alpha) \in F(\Lambda ^2 A_{1} ^{\times})^{\circ},$ with $\alpha \in I_{1}$ and $a \in \underline{A}_{1}.$ In order to compute its image under $h_{f}(\tau_1,\tau_2).$ We first need to lift it to an element in $F(\Lambda^2\hat{B}_{1}^{\times})^{\circ}.$ Because of our choice of the rings $\hat{B}_{i},$ $\alpha \in I_1\subseteq \hat{I}_{1}$ and $a \in \underline{A}_{1} \subseteq \hat{B}_1.$ The element
$$
\tilde{\gamma}:=e^{\alpha} \wedge e^{a\alpha}-\frac{1}{2}e^{a\alpha^2}\wedge a
$$
is a lifting of the above element and has the property that its image under the map $\log d log $ is equal to  
$$
\alpha \cdot d(a\alpha)-\frac{1}{2}a\alpha^2\frac{da}{a}=\frac{1}{2}\alpha^2da+a\alpha d \alpha=\frac{1}{2}d(a\alpha^2)=0 \in (\Omega^{1} _{\hat{B}_{1}}/d(\hat{B}_{1}))^{\circ}.
$$
Therefore we need to apply $h_{\hat{f}}(\hat{\tau}_{1},\hat{\tau}_{2})$ to $\tilde{\gamma}.$  We first compute 
$$
( \log ^{\circ} _{\hat{\tau}_1} \wedge \underline{d} log)(\tilde{\gamma})=\alpha \underline{d}(a\alpha)-a\alpha \underline{d}(\alpha)=0,
$$
 since both $\log ^{\circ} _{\hat{\tau}_{1}}(a)$ and $\underline{d}(a)$ are 0. 
 
 Therefore, $h_{\hat{f}}(\hat{\tau}_{1},\hat{\tau}_{2})(\tilde{\gamma})=-3
 \underline{d}^{-1}( \log ^{\circ} _{\hat{\tau}_2} \wedge \underline{d} log) (\hat{f}(\tilde {\gamma})) \in \hat{J}_{2}/\hat{J}_{2}^2=\hat{I}_{2}^2/\hat{I}_{2}^4 
 .$ The value of the map  $(\log ^{\circ} _{\hat{\tau}_2} \wedge \underline{d} log)\circ \hat{f}$ on the element  $-\frac{1}{2}e^{a\alpha^2}\wedge a$ is equal to the sum of 
 $$
 -\frac{1}{2}(\hat{f}_0(a\alpha^2)\frac{\underline{d}(\hat{f}_1(a))}{\hat{f}_0(a)}-\frac{\hat{f}_1(a)}{\hat{f}_0(a)}\underline{d}(\hat{f}_0(a\alpha^2)))=-\frac{1}{2}(\hat{f}_0(\alpha)^2\underline{d}(\hat{f}_1(a))-\hat{f}_1(a)\underline{d}(\hat{f}_0(\alpha)^2))
 $$
 and an element in $\hat{I}^3 _2 \underline{d}(\hat{I}_2).$
 
 Similarly, the value of the same map on the term $e^{\alpha} \wedge e^{a\alpha}$ is equal to the sum of 
 \begin{align*}
 &(\hat{f}_{0}(\alpha)+ \hat{f}_{1}(\alpha))\underline{d}(\hat{f}_{0}(a\alpha)+ \hat{f}_{1}(a\alpha))-(\hat{f}_{0}(a\alpha)+ \hat{f}_{1}(a\alpha))\underline{d}(\hat{f}_{0}(\alpha)+ \hat{f}_{1}(\alpha))\\
 =&\hat{f}_{0}(\alpha)\underline{d}(\hat{f}_{1}(a)\hat{f}_{0}(\alpha) )-\hat{f}_{1}(a)\hat{f}_{0}(\alpha)\underline{d}(\hat{f}_{0}(\alpha))=\hat{f}_{0}(\alpha)^2\underline{d}(\hat{f}_{1}(a))
 \end{align*}
and an element in $\hat{I}^3 _2 \underline{d}(\hat{I}_2).$ 

Combining these two statements we deduce that 
 $$
 h_{\hat{f}}(\hat{\tau}_{1},\hat{\tau}_{2})(\tilde{\gamma})=-\frac{3}{2}\hat{f}_{0}(\alpha)^2\hat{f}_{1}(a).
 $$

(ii) Suppose that we start with the element 
$$
(1+\frac{\beta}{b-1}) \wedge b-(1+\frac{\beta}{b})\wedge (b-1) \in F(\Lambda ^2 A_{1} ^{\times})^{\circ},
$$
with $\beta \in I_{1}$ and $b \in \underline{A}_{1} ^{\flat} \subseteq A_{1}.$ Analogous to the above case, the element 
$$
\tilde{\delta}:=e^{\beta/(b-1)} \wedge b - e^{\beta/b}\wedge (b-1) \in F(\Lambda ^2 \hat{B}_{1} ^{\times})^{\circ}
$$
lifts  the above element. 
 Since $( \log ^{\circ} _{\hat{\tau}_1} \wedge \underline{d} log)(\tilde{\delta})=0,$ we have $$
 h_{\hat{f}}(\hat{\tau}_{1},\hat{\tau}_{2})(\tilde{\delta})=-3
 \underline{d}^{-1}( \log ^{\circ} _{\hat{\tau}_2} \wedge \underline{d} log) (\hat{f}(\tilde {\delta})) \in \hat{J}_{2}/\hat{J}_{2}^2=\hat{I}_{2}^2/\hat{I}_{2}^4 
 .
 $$
 By an explicit computation, we see that 
 $( \log ^{\circ} _{\hat{\tau}_2} \wedge \underline{d} log) (\hat{f}(\tilde {\delta}))$ is the sum of 
 $$
 \frac{1}{2} \frac{\hat{f}_{1}(b)^2 \underline{d}(\hat{f}_{0}(\beta)) }{(\hat{f}_{0}(b)(\hat{f}_{0}(b)-1))^2}+\frac{\hat{f}_{0}(\beta)\hat{f}_{1}(b) \underline{d}(\hat{f}_{1}(b)) }{(\hat{f}_{0}(b)(\hat{f}_{0}(b)-1))^2}
 $$
 and an element in $\hat{I}^3 _2 \underline{d}(\hat{I}_2).$ This implies that 
  $$
 h_{\hat{f}}(\hat{\tau}_{1},\hat{\tau}_{2})(\tilde{\delta})=-\frac{3}{2}\frac{\hat{f}_{1}(b)^2 \hat{f}_{0}(\beta) }{(\hat{f}_{0}(b)(\hat{f}_{0}(b)-1))^2}.
 $$

\subsubsection{Description of $h_{f}(\tau_1,\tau_2)$} Using the notation above, let us denote $\hat{f}_1$ by $\hat{\theta}.$ We define a map 
$$
h_{\hat{\theta}}: (\Lambda^2 A_{1} ^{\times})^{\circ} \to \hat{I}_2 ^3 / \hat{I}_2 ^4
 $$
 as follows. 
 
For  $\alpha,\, \beta \in I_{1} \subseteq \hat{I}_1,$ note that $f(\alpha),\, f(\beta) \in I_{2} \subseteq \hat{I}_2 ,$ and we define 
$$
h_{\hat{\theta}}((1+\alpha)\wedge (1+\beta)):=f(\alpha) \hat{\theta} (\beta)-\hat{\theta}(\alpha)f(\beta)
.$$ 
For $a \in \underline{A}_1 ^{\flat} \subseteq \hat{B}_{1}$ and $\alpha \in I_{1}\subseteq \hat{I}_{1},$ we let 
$$
h_{\hat{\theta}}((1+\alpha)\wedge \tau_{1}(a)):=-\hat{\theta}(\alpha)\frac{\hat{\theta}(a)}{\underline{f}(a)},$$
where we view $\underline{f}(a) \in \underline{A}_{2} \subseteq \hat{B}_2.$ The additivity in the second component follows from the Leibniz formula $\hat{\theta}(ab)=\underline{f}(a)\hat{\theta}(b)+\hat{\theta}(a)\underline{f}(b)$ in $\hat{B}_{2}/\hat{I}_2 ^2,$ for $a,$ $b \in \underline {A}_{1}.$ 

The following proposition then gives the explicit expression for $h_{f}(\tau_1,\tau_2)$ we are looking for. 
\begin{proposition}\label{equation h}
The map $h_{f}(\tau_1,\tau_2)$ coincides with the restriction of $-\frac{3}{2}h_{\hat{\theta}}$ to $F(\Lambda^2 A_{1} ^{\times})^{\circ}.$ In particular, $h_{\hat{\theta}}$ does not depend on the liftings $\hat{f},$ $\hat{\tau}_1,$ and $\hat{\tau}_2$ on $F(\Lambda^2 A_{1} ^{\times})^{\circ}.$ 
\end{proposition}

\begin{proof}
By the definition of $h_{f}(\tau_1,\tau_2),$ we only need to show that $h_{\hat{f}}(\hat{\tau}_{1},\hat{\tau}_{2})$ coincides with $-\frac{3}{2}h_{\hat{\theta}}$ on $F(\Lambda^2 A_{1} ^{\times})^{\circ}.$ By the formulas for $h_{\hat{f}}(\hat{\tau}_{1},\hat{\tau}_{2})$ in \textsection \ref{Description of hat h} above, we see that these two functions agree on elements of the form $(1+\alpha)\wedge (1+a\alpha)$ and $(1+\frac{\beta}{b-1})\wedge b -(1+\frac{\beta}{b})\wedge (b-1).$ In order to finish the proof, we only need to prove that these elements generate   $F(\Lambda^2 A_{1} ^{\times})^{\circ}.$  Since $F(\Lambda^2 A_{1} ^{\times})^{\circ}={\rm im} (\delta^{\circ})$ and 
$$
\delta^{\circ}(\{ a+\alpha \}_{2}-\{ a)=(a-1)\wedge (1+\frac{\alpha}{a})-a\wedge (1+\frac{\alpha}{a-1})+(1+\frac{\alpha}{a-1})\wedge (1+\frac{\alpha}{a}),
$$
these elements indeed generate $F(\Lambda^2 A_{1} ^{\times})^{\circ}.$
\end{proof}

\section{Regulator map and the Proof of the  Theorem} 
In this section, we continue to assume that $\underline{X}$ is smooth over $k$ and that the conormal sheaf of the imbedding  $\underline{X}\hookrightarrow X$ is locally free on $\underline{X}.$ 
Let $\{U_i \}_{i\in I}$ be an open affine cover of $X.$ Suppose that $\tau_i$ are local splittings of $\underline{U}_i \hookrightarrow U_i.$ Let $\{ a_{ij} \}_{i,j \in I} $ be local sections of $\underline{B}_2 ^{\circ}$ on $U_{ij}$ and $\{ b_{i} \}_{i \in I}$ be local sections of $F(\Lambda^{2} \mathcal{O}_{X} ^{\times})^{\circ}$  on $U_{i}$ such that $\delta (a_{ij})=b_{j}|_{U_{ij}}-b_{i}|_{U_{ij}},$ and $a_{jk}|_{U_{ijk}}-a_{ik}|_{U_{ijk}}+a_{ij}|_{U_{ijk}}=0.$ 

Consider the element 
$$
\gamma_{ij}:=\ell i_{2,\tau_i}(a_{ij})+h(\tau_i,\tau_j)(b_j) \in \underline{D}_1(U_{ij}).
$$

Since on $U_{ijk},$ $a_{jk}-a_{ik}+a_{ij}=0,$ and $h(\tau_j,\tau_k)-h(\tau_i,\tau_k)=h(\tau_j,\tau_i)$ , we have the following equalities,   $\gamma_{jk}-\gamma_{ik}+\gamma_{ij}=$
\begin{align*}
&  \ell i_{2,\tau_j}(a_{jk}) -\ell i_{2,\tau_i}(a_{ik})+\ell i_{2,\tau_i}(a_{ij})+h(\tau_j,\tau_k)(b_k)-h(\tau_i,\tau_k)(b_k)+h(\tau_i,\tau_j)(b_j)\\
 =&  \ell i_{2,\tau_j}(a_{jk})- \ell i_{2,\tau_i}(a_{jk})+h(\tau_j,\tau_i)(b_k)+h(\tau_i,\tau_j)(b_j)\\
 =&h(\tau_i,\tau_j)(\delta(a_{jk}))+h(\tau_j,\tau_i)(b_k)+h(\tau_i,\tau_j)(b_j)=h(\tau_i,\tau_j)(b_k-b_j)+h(\tau_j,\tau_i)(b_k)+h(\tau_i,\tau_j)(b_j)\\
 =&0.
\end{align*}
Therefore $\{\gamma_{ij} \}_{i, j \in I}$ defines a cocycle. If $\{\tau_i ' \}_{i \in I}$ is another set of splittings and $\{\gamma_{ij} ' \}_{i, j \in I}$ the corresponding cocycle, then we have $\gamma_{ij}' -\gamma_{ij}=$
\begin{align*}
  & \ell i_{2,\tau_i'}(a_{ij})+h(\tau_i',\tau_j')(b_j)- \ell i_{2,\tau_i}(a_{ij})-h(\tau_i,\tau_j)(b_j)\\
   =&h(\tau_i,\tau_i')(\delta(a_{ij}))+h(\tau_i',\tau_j')(b_j)-h(\tau_i,\tau_j)(b_j)\\
   =&h(\tau_i,\tau_i')(b_j-b_i)+h(\tau_i',\tau_j')(b_j)-h(\tau_i,\tau_j)(b_j)\\
   =&h(\tau_j,\tau_j')(b_j)-h(\tau_i,\tau_i')(b_i),
\end{align*}
which is the co-boundary of the element $\{ h(\tau_i,\tau_i')(b_i)\}_{i \in I}.$ 
In order to finish the regulator construction we need to show that the boundaries go to boundaries. Therefore if $\{ a_{i}\}_{i\in I}$ is a collection of sections of $\underline{B}_{2}$ over $U_{i},$ the map sends the boundary of this element to 
$$
\ell i_{2,\tau_i}(a_j-a_i)+h(\tau_i,\tau_j)(\delta(a_j))=\ell i_{2,\tau_i}(a_j) -\ell i_{2,\tau_i}(a_i)+\ell i_{2,\tau_j}(a_j)-\ell i_{2,\tau_i}(a_j)
$$
on $U_{ij},$ which is the boundary of $\{\ell i _{2,\tau_i}(a_i) \}_{i \in I}.$ 

This defines the map 
$$
\rho_2: {\rm H}^2(X,F\Gamma_{X}^{\circ}(2)) \to {\rm H}^1(X,D_{1}(\mathcal{O}_X)),
$$
we were looking for. The map $\rho_1,$ immediately follows from the surjective map of complexes 
$$
\Gamma_{X} ^{\circ}(2) \to K_{2} ^{M}(\mathcal{O}_{X})^{\circ}_{\mathbb{Q}}[-2],
$$
together 
with the identification $K_{2} ^{M}(\mathcal{O}_{X})^{\circ}_{\mathbb{Q}}=(\Omega^1 _{X}/d \mathcal{O}_{X})^{\circ}.$ 

{\it Proof of Theorem \ref{thm1}.}
Since the  map of complexes ${\rm ker} (\delta^{\circ})[-1] \to F\Gamma_{X}^{\circ}(2)$ is a quasi-isomorphism, we have an isomorphism
${\rm H}^{1}(X,{\rm ker}(\delta^{\circ}))  \xrightarrow{\sim} {\rm H}^2(X,F\Gamma_{X}^{\circ}(2)).$  By Corollary \ref{localiso},  for any choice $\tau$ of a local splitting of $\underline{U} \hookrightarrow U,$  there is an isomorphism 
\begin{eqnarray}\label{localmap}
\ell i_{2,\tau}: {\rm ker}(\delta^{\circ})|_{U} \to D_1(\mathcal{O}_{X})|_{U}. 
\end{eqnarray}
If $\tau'$ is a different splitting $\ell i_{2,\tau '}=\ell i_{2,\tau}$ on ${\rm ker}(\delta ^{\circ})$ by (\ref{eqhom}). Therefore, the local isomorphism  (\ref{localmap}) is independent of the choice of the splitting and gives a global isomorphism ${\rm ker}(\delta^{\circ})\to D_{1}(\mathcal{O}_X).$ This proves that $\rho_2$ is an isomorphism. 

Suppose that we have a map $f:X_{2}\to X_{1}$ of $k$-schemes. The functoriality of $\rho_1$ with respect to $f$ 
is clear, whereas that of $\rho_2$ can be deduced easily by using the homotopy maps $h_{f}(\tau_1.\tau_2),$ for choices of, not necessarily compatible, splittings on $X_{1}$ and $X_2.$   
\hfill $\Box$

\section{Crystalline Deligne-Vologodsky complex}\label{vol} 

There is another, somewhat more transcendental, complex of Zariski sheaves  which is expected to compute the motivic cohomology of the infinitesimal part of the motivic cohomology of  $X.$ This is the crystalline version of the Deligne complex as defined by Vologodsky \cite{vol}. In this section, we assume that $X/k_2$ is smooth. 

Let $\mathcal{J}_{X}$ denote the subsheaf of the crystalline structure sheaf $\mathcal{O}_X$ which associates to  an infinitesimal thickening $U \hookrightarrow T$ of an open subset  $U$ of $X,$ the kernel of the map $\mathcal{O}(T) \to \mathcal{O}(U).$ Let $\pi$ denote the natural map from the crystalline site to the Zariski site over $X.$ For $1\leq i,$ let  
$$
\underline{\mathcal{D}}_{X} ^{\circ}(i):={\rm Cone} ({\rm R}\pi (\mathcal{O}_{X}/\mathcal{J}_X ^i) \to {\rm R}\pi (\mathcal{O}_{\underline{X}}/\mathcal{J}_{\underline{X}} ^i) )[-2]
$$
denote the complex defined by Vologodsky \cite[\textsection 7]{vol}.  The cohomology groups
$$
{\rm H}^{*}_{crys} (X,\mathbb{Q}_{D}(i))^{\circ}:={\rm H}^{*}(X,\underline{\mathcal{D}}_{X} ^{\circ}(i)) 
$$
of this complex are the crystalline analog of the Deligne cohomology groups. Using Goodwillie's theorem \cite{hess}, Vologodsky proves that the Chern character gives  an abstract isomorphism $K_{m} ^{\circ}(X)_{\mathbb{Q}} ^{(i)} \simeq {\rm H}^{2i-m}(X,\underline{\mathcal{D}}_{X} ^{\circ}(i)).$ 

Note that  $\underline{\mathcal{D}}_{X} ^{\circ}(1)=\mathcal{O}_{X} ^{\circ}[-1],$ the sheaf $\mathcal{O}_{X} ^{\circ}$ in degree 1. We will only be concerned with this complex when $i=2.$ When there is a imbedding of $X$ into a smooth scheme $P$  such that the sheaf of ideals of $X$ in $P$ is $J$ then $\underline{\mathcal{D}}_{X} ^{\circ}(2)$ is quasi-isomorphic to the complex 
$$
\mathcal{O}_{\hat{P}} ^{\circ} /J^2 \to (\Omega^{1} _{\hat{P}}) ^{\circ}/J\cdot \Omega^{1} _{\hat{P}}
$$
concentrated in degrees 1 and 2, where $\Omega_{\hat{P}} ^{i\circ}$ is the kernel of the map $\Omega_{\hat{P}} ^{i} \to \Omega_{\underline{X}} ^{i}, $ and $\hat{P}$ is the completion of $P$ along $X.$

Let $\mathcal{H}^{*}(\underline{\mathcal{D}}_{X} ^{\circ}(i))$ denote the Zariski sheaves obtained by taking the cohomology sheaves of the complex $\underline{\mathcal{D}}_{X} ^{\circ}(i).$ By choosing local splittings, one sees that $\mathcal{H}^1(\underline{\mathcal{D}}_{X} ^{\circ}(2))=D_{1}(\mathcal{O}_X)$ and $\mathcal{H}^2(\underline{\mathcal{D}}_{X} ^{\circ}(2))=(\Omega^1 _{X}/d \mathcal{O}_{X})^{\circ}.$ Therefore, Vologodsky's theorem also gives maps from $K_{2}(X)^{(2)} _{\mathbb{Q}}$ similar to the one that we constructed above. We do not currently know how to compare these maps, but will consider this problem in a future work.




\end{document}